\documentclass[11pt,a4paper]{article}
\usepackage{tikz}
\usepackage{subfigure}
\usepackage[english]{babel}
\usepackage{graphicx}
\usepackage{extarrows}
\usepackage{setspace}

\usepackage{indentfirst, latexsym, bm, amssymb}

\begin{document}

\newtheorem{theorem}{Theorem}[section]
\newtheorem{corollary}[theorem]{Corollary}
\newtheorem{definition}[theorem]{Definition}
\newtheorem{conjecture}[theorem]{Conjecture}
\newtheorem{question}[theorem]{Question}
\newtheorem{lemma}[theorem]{Lemma}
\newtheorem{proposition}[theorem]{Proposition}
\newtheorem{example}[theorem]{Example}
\newenvironment{proof}{\noindent {\bf
Proof.}}{\hfill $\square$\par\medskip}
\newcommand{\remark}{\medskip\par\noindent {\bf Remark.~~}}
\newcommand{\pp}{{\it p.}}
\newcommand{\de}{\em}

\title{Extremal graphs for odd wheels\thanks{This work is supported by the National Natural Science Foundation of China (No. 11901554) and Science and Technology Commission
of Shanghai Municipality (No. 18dz2271000, 19jc1420100)}}
\author{ Long-Tu Yuan  \\
{\small School of Mathematical Sciences and Shanghai Key Laboratory of PMMP}\\
{\small East China Normal University}\\
{\small 500 Dongchuan Road, Shanghai, 200240, P.R. China  }\\
{\small Email: ltyuan@math.ecnu.edu.cn}\\
}
\date{}
\maketitle
\begin{abstract}
For a graph $H$, the Tur\'{a}n number of $H$, denoted by ex$(n,H)$, is the maximum number of edges of an $n$-vertex  $H$-free graph. Let $g(n,H)$ denote the maximum number of edges not contained in any monochromatic copy of $H$ in a $2$-edge-coloring of $K_n$. A wheel $W_m$ is a graph formed by connecting a single vertex to all vertices of a cycle of length $m-1$.  The Tur\'{a}n number of $W_{2k}$ was determined by Simonovits in the 1960s.  In this paper, we determine ex$(n,W_{2k+1})$ when $n$ is sufficiently large. We also show that, for sufficiently large $n$,  $g(n,W_{2k+1})=\mbox{ex}(n,W_{2k+1})$ which confirms a conjecture posed by Keevash and Sudakov for odd wheels.
\end{abstract}


\vskip 0.5cm

\section{Introduction}
Let $H$ be a given graph. We say a graph is $H$-free if it does not contain $H$ as a subgraph. The {\it Tur\'{a}n number} of $H$, denoted by ex$(n,H)$, is the maximum number of edges in an  $n$-vertex $H$-free graph.  An $n$-vertex graph $G$ is an \textit{extremal graph} for $H$ if $G$ is $H$-free with ex$(n,H)$ edges. The famous Tur\'{a}n Theorem \cite{turan1941} states  that the unique extremal $n$-vertex graph for the complete graph $K_{p+1}$ is the balanced complete $p$-partite graph $T(n,p)$.  Let   $t(n,p)$ be the number of edges of $T(n,p)$. For the Tur\'{a}n number of the balanced complete $p$-partite graph, Erd\H{o}s and Stone \cite{erdHos1946} proved the following well-known theorem.

\begin{theorem}(Erd\H{o}s and Stone \cite{erdHos1946}).\label{erdos-stone}
For all integers $p\geq 2$ and $N\geq 1$, and every $\epsilon >0$, there exists an integer $n_0$ such that every graph on $n\geq n_0$ vertices and with at least
\[t(n,p-1)+\epsilon n^2\]
edges contains $T(N,p)$ as a subgraph.
\end{theorem}

In general, Erd\H{o}s and Simonovits \cite{erdHos1966} proved the following theorem.
\begin{theorem}(Erd\H{o}s and Simonovits \cite{erdHos1966}).\label{p+1 chromatic}
Let $\chi(H)$ be the chromatic number of a given graph $H$. For sufficiently large $n$, we have
 \[\emph{ex}(n,H)=\left(1-\frac{1}{\chi(H)-1}\right){n \choose 2}+o(n^2).\]
\end{theorem}

Although the Erd\H{o}s-Stone-Simonovits theorem gives us the asymptotic values of Tur\'{a}n numbers of non-bipartite graphs, in general, it is a challenge to determine the exact values of Tur\'{a}n numbers of non-bipartite graphs.

Let $H$ be a given graph with $\chi(H)=p+1$. If there is an edge $e$ such that $\chi(H-e)=p$, then we say that $H$ is {\it edge-critical.}  Simonovits \cite{Simonovits1968} proved the following result.

\begin{theorem}(Simonovits \cite{Simonovits1968}). \label{Chromatic Critical Edge}
Let $H$ be an edge-critical graph with $\chi(H)=p+1$. Then $T(n,p)$ is the unique  extremal graph for $H$, provided $n$ is sufficiently large.
\end{theorem}

Denote by $C_k$ the cycle of length $k$. A {\it wheel} $W_m$ is a graph formed by connecting a single vertex to all vertices of $C_{m-1}$. If $m$ is odd, then we say that $W_m$ is an odd  wheel. We can define even wheels similarly.
For even $m$, we have $\chi(W_{m})=4$ and $\chi(W_m-e)=3$ for each $e$ in the cycle. By Theorem~\ref{Chromatic Critical Edge}, the unique extremal graph for $W_m$ is $T(n,3)$.   There are only  specific families of graphs whose  extremal graphs are known, see \cite{chen2003,erdHos1995,Liu2013,Simonovits1974.1,Yuan2019+}.

In this paper, we will determine the exact values of Tur\'{a}n numbers of odd wheels. A theorem (Theorem 1 in \cite{Simonovits1974}) of Simonovits states that if the decomposition family (see Section~\ref{decomposition family}) of the forbidden graph $H$ contains a linear forest\footnote{a linear forest is a forest in which each component is a path}, then the extremal graphs for $H$ have simple and symmetric structures (as the theorem is quite complicated, we refer the interested readers to \cite{Simonovits1974} for more information). The decomposition families of  all the forbidden graphs mentioned in \cite{chen2003,erdHos1995,Liu2013,Simonovits1974.1,Yuan2019+} contains a linear forest. As the decomposition family  of $W_{2k+1}$ does not contain a linear forest, it is interesting to determine the exact values of Tur\'{a}n numbers of $W_{2k+1}$.

There is a partial result for the Tur\'{a}n numbers of odd wheels. Dzido and Jastrz\c{e}bski \cite{Dzido2018} determined ex$(n,W_5)$ and ex$(n,W_7)$ for all value of $n$. They also established  a lower bound on ex$(n,W_{2k+1})$ for $k\geq 4$.   We will  show that this lower bound  is the exact value of ex$(n,W_{2k+1})$ for infinitely many $n$. Let $n$, $n_0$ and $n_1$ be integers. We define  $f(n,k)$ as following:
\begin{eqnarray*}
f(n,k)=\max\left\{ n_0n_1+ \left\lfloor\frac{(k-1)n_0}{2}\right\rfloor   +1:n_0+n_1=n\right\}.
\end{eqnarray*}

Our first result is the following theorem.

\begin{theorem}\label{main}
Let $k\geq 2$ and $W_{2k+1}$ be a wheel on $2k+1$ vertices. Then, for large $n$, we have
$$\emph{ex}(n,W_{2k+1})=\left\{
\begin{array}{lll}
(\lceil\frac{n}{2}\rceil+1)\lfloor\frac{n}{2}\rfloor  &\mbox{ if }& k=2;\\
   f(n,k)  &\mbox{ if } & k\geq 3.
\end{array}\right.$$
\end{theorem}

Very recently, based on Theorem~\ref{main}, Xiao and Zamora~\cite{Xiao} determine the Tur\'{a}n number of vertex-disjoint copies of odd wheels.

For a given graph $H$, let $g(n,H)$ denote the maximum number of edges not contained in any monochromatic copy of $H$ in a $2$-edge-coloring of $K_n$. If we color the edges of an extremal $n$-vertex graph for $H$ with the same color and color other edges with the other color, then we can see that $g(n,H)\geq \mbox{ex}(n,H)$ for any $H$ and $n$. In 2004, Keevash and Sudakov showed in  \cite{Keevash-Sudakov-2004} that this lower bound is tight for sufficiently large $n$ if $H$ is edge-critical or a cycle of length four. Hence, they posed the following conjecture.

\begin{conjecture}(Keevash and Sudakov \cite{Keevash-Sudakov-2004}).\label{Keevash-Sudakov-conjecture} Let $H$ be a given graph. If $n$ is sufficiently large, then
$$g(n,H)=\emph{ex}(n,H).$$
\end{conjecture}

In 2017, Ma \cite{Ma-2017} confirmed Conjecture~\ref{Keevash-Sudakov-conjecture} for an infinite family of bipartite graphs. Later, Liu, Pikhurko, and Sharifzadeh \cite{Liu-Pikhurko-Sharifzadeh-2019} extended Ma's result to a larger family of bipartite graphs and proved an upper bound for all bipartite graphs.

For non-bipartite graphs, the following theorem confirms Conjecture~\ref{Keevash-Sudakov-conjecture} for odd wheels.

\begin{theorem}\label{MAIN}
	Let $n$ be sufficiently large and $k\geq 2$. Then
	\begin{equation*}
	g(n,W_{2k+1})=\emph{ex}(n,W_{2k+1}).
	\end{equation*}
\end{theorem}

The organisation of this paper is as follows. In Section 2, we introduce the extremal graphs for odd wheels. In Section 3, we present several lemmas and useful definitions. In Section 4, we will prove Theorem~\ref{main}. In Section 5, we will first prove some lemmas and then sketch the proof Theorem~\ref{MAIN}.

\section{Extremal graphs and extremal colorings for odd wheels}\label{extrmel-graphs-for-wheel}
For a graph with $n$ vertices, let $d_1  \geq d_2 \geq \ldots \geq  d_n$ be its
  {\it degree sequence} in a non-increasing order. A {\it graphic sequence} is a list of non-negative numbers that is the degree sequence of some simple graph and we say that this  sequence is {\it graphic}.  Havel \cite{havel1955} and  Hakimi \cite{hakimi1962} (also see \cite{West2000}) proved the following.

\begin{theorem}\label{theorem for graph sequence}(Havel \cite{havel1955} and Hakimi \cite{hakimi1962}).
For $n > 1$, an integer list $\mathbf{d}$
of size $n$ is graphic if and only if $\mathbf{d}^1$ is graphic, where $\mathbf{d}^1$ is obtained from $d$
by deleting its largest element $d_1$ and subtracting $1$ from its $d_1$ next largest
elements. The only $1$-element graphic sequence is $d_1=0$.

\end{theorem}

{\it A nearly $(k-1)$-regular graph} is a graph such that each vertex has degree $k-1$ but one vertex has degree $k-2$.

\begin{proposition}\label{existence of k-1 regular graph}
Let $n\geq k$. When $n$ is odd and $k$ is even, there exist nearly $(k-1)$-regular graphs on $n$ vertices.  Otherwise, there exist $(k-1)$-regular graphs on $n$ vertices.
\end{proposition}
\begin{proof}
 Let $\mathbf{d}^0=(k-1,k-1,\ldots,k-1,k-1-i)$ be a sequence of  $n$ non-negative integers, where $i=1$ if $k$ is even and $n$ is odd, and $i=0$ otherwise.

 We say that a graph sequence $\mathbf{d}$ is {\it balanced} if the difference between any two elements of it is at most one (the only $1$-element graphic sequence is balanced).

 We define $\mathbf{d}^{i+1}$ from $\mathbf{d}^i$ by deleting its largest element $d^i_1$ of $\mathbf{d}^i$ and subtracting $1$ from its $d^i_1$ next largest elements. Let $s_i$ be the sum of the elements of $\mathbf{d}^{i}$. Thus it is easy to see that $\mathbf{d}^i$ is balanced and $s_i$ is an even number for each $i\in\{0,1,\ldots,n-1\}$. By Theorem~\ref{theorem for graph sequence}, it is sufficient to show that $\mathbf{d}^{n-1}$ is a graphic sequence, that is,  the only element of $\mathbf{d}^{n-1}$ is $0$.

\medskip

\noindent{\bf Claim.} $s_i\leq 2{n-i \choose 2}$ for each $i\in \{0,\ldots,n-2\}$.

\medskip

\begin{proof}
We prove the claim by induction on $i$. Obviously the claim holds for $i=0$. Suppose that the claim holds for $i-1$. Then we have $s_{i-1}\leq2{n-i+1 \choose 2}$. Let $d_{1}^{i-1}$ be a largest element of $\mathbf{d}^{i-1}$. Note that $\mathbf{d}^j$ is balanced for each $j\in\{0,1,\ldots,n-1\}$. We have
\begin{eqnarray*}
s_i&=&s_{i-1}-d_1^{i-1}\leq s_{i-1}- \left\lceil\frac{s_{i-1}}{n-i}\right\rceil\\
&=& \left\lfloor\left(1-\frac{1}{n-i}\right)s_{i-1}\right\rfloor\leq \left\lfloor\left(1-\frac{1}{n-i}\right) 2{n-i +1\choose 2}\right\rfloor \\
 &= &2{n-i \choose 2}.
\end{eqnarray*}
We finish the proof of the claim.\end{proof}

Applying the claim for $i=n-2$, the sum of the elements of $\mathbf{d}^{n-2}$ is at most $2$. Since $s_{n-2}$ is an even number, we have $s_{n-2}=0$ or $2$. Thus the only element of $\mathbf{d}^{n-1}$ is $0$. The proof is complete.\end{proof}

Let $S_{k}$ be the star on $k$ vertices and $P_{k}$ be the path on $k$ vertices.

\begin{proposition}\label{k-1 regular no C2k}
Let $n\geq 2k$. Then
$$\emph{ex}(n,\{S_{k+1},P_{2k-1}\})=\left\lfloor\frac{(k-1)n}{2}\right\rfloor.$$
\end{proposition}
\begin{proof} It follows from Proposition~\ref{existence of k-1 regular graph} that there exist $(k-1)$-regular or nearly $(k-1)$-regular graphs on $m\geq k$ vertices. Note that a graph containing $P_{2k-1}$ as a subgraph has at least $2k-1$ vertices, there are $\{S_{k+1},P_{2k-1}\}$-free graphs on $k\leq m\leq 2k-2$ vertices with $\left\lfloor\frac{(k-1)m}{2}\right\rfloor$ edges. If $n$ is divisible by $k$, then the disjoint union of $\tfrac{n}{k}$ cliques of size $k$ is $\{S_{k+1},P_{2k-1}\}$-free and has $\tfrac{(k-1)n}{2}$ edges. Next, we suppose that $n=sk+r$, where $1 \leq r \leq k-1$. The assumption $n \geq 2k$ gives $s \geq 2$. If $k$ is odd, then we write $n=(s-2)k+2k+r=(s-2)k+t_1+t_2$, where $k \leq t_1, t_2 \leq 2k-2$. We apply Proposition \ref{existence of k-1 regular graph} to construct an $\{S_{k+1},P_{2k-1}\}$-free graph which consists of $s-2$ cliques of size $k$ and two $(k-1)$-regular graph with $t_1$ vertices and $t_2$ vertices respectively. We can easily check this graph has
$\tfrac{(k-1)n}{2}$ edges. If $k$ is even, then we have two cases depending on the parity of $n$. For the case of $n$ being even, we can prove it by repeating the argument for the case where $k$ is odd by noting that both $t_1$ and $t_2$ are even.   For the case of  $n$ being odd, we note $2k+r$ is odd and $2k+1 \leq 2k+r \leq 3k-1$. We can write $2k+r=t_1+t_2$,
 here $k \leq t_1,t_2 \leq 2k-2$ and one of them is odd, say $t_2$. Similarly, we construct a graph consisting of $(s-2)$ disjoint copies of $K_k$, a $(k-1)$-regular graph with $t_1$ vertices and a nearly $(k-1)$-regular graph with $t_2$ vertices. It is easy to verify that this graph has
 $ \left\lfloor\tfrac{(k-1)n}{2}\right\rfloor$ edges. The proof is complete.\end{proof}

We are ready to introduce the extremal graphs for $W_{2k+1}$. Denote by $\mathcal{U}_n^k$ the class of $(k-1)$-regular graphs or nearly $(k-1)$-regular graphs on $n$ vertices. Denote by $\mathcal{U}^k_n(P_{2k-1})$ the subset of $\mathcal{U}_n^k$ such that any graph in it is $P_{2k-1}$-free. It follows from Proposition~\ref{k-1 regular no C2k} that $\mathcal{U}^k_n(P_{2k-1})$ is not empty for $n\geq 2k$. Let $n_0\geq n_1\geq 2$ and $n_0\geq2 k $. Denote by $K_{n_0,n_1}(\mathcal{U}^k_{n_0}(P_{2k-1}),K_2)$  the class of graphs obtained by taking a complete bipartite graph $K_{n_0,n_1}$ and embedding a graph from $\mathcal{U}^k_{n_0}(P_{2k-1})$ into the larger partite set   and embedding an edge into the smaller partite set. Let $\mathcal{K}_n^k$ be the subset of $K_{n_0,n_1}(\mathcal{U}^k_{n_0}(P_{2k-1}),K_2)$ which consists of graphs with $f(n,k)$ edges. It is easy to check that graphs in $\mathcal{K}_n^k$  are $W_{2k+1}$-free. We will show that the graphs in $\mathcal{K}_n^k$ are extremal graphs for $W_{2k+1}$, provided $n$ is sufficiently large.

For Theorem~\ref{MAIN}, we will show that one of the extremal colorings is obtained by taking a blue copy of $G_n\in\mathcal{K}_n^k$ and coloring the complement of $G_n$ red. Other extremal colorings can be obtained by changing some blue edges between the classes of $G_n$ to red edges such that those new red edges are not contained in a red copy of $W_{2k+1}$.

\medskip


\section{Lemmas for Theorems~\ref{main} and~\ref{MAIN}}

\subsection{Lemma of progressive induction.}

Simonovits \cite{Simonovits1968} introduced the {\it progressive induction} which is  a powerful tool to solve extremal problems in graph theory. For examples, see \cite{han2019,keevash2004}.
\begin{lemma}(Simonovits \cite{Simonovits1968}).\label{progrssion induction}
Let $\mathfrak{U}=\cup_{i=1}^{\infty}\mathfrak{U}_i$ be a set of given elements, such that $\mathfrak{U}_i$ are disjoint subsets of $\mathfrak{U}$. Let $B$ be a condition or property defined on $\mathfrak{U}$ (i.e. the elements of $\mathfrak{U}$ may satisfy or not satisfy $B$). Let $\phi(a)$ be a function defined  on $\mathfrak{U}$ such that $\phi(a)$ is a non-negative integer and\\
(a) if $a$ satisfies $B$, then $\phi(a)=0$.\\
(b) there is an $M_0$ such that if $n>M_0$ and $a\in \mathfrak{U}_n$ then either $a$ satisfies $B$ or there exist an $n^{\prime}$ and an $a^{\prime}$ such that
\[\frac{n}{2}<n^{\prime}<n, a^{\prime}\in \mathfrak{U}_{n^{\prime}} \mbox{ and } \phi(a)<\phi(a^{\prime}).\]
Then there exists an $n_0$ such that if $n>n_0$, every $a\in \mathfrak{U}_n$  satisfies $B$.
\end{lemma}

\remark In our proof of Theorem~\ref{main}, $\mathfrak{U}_n$ is the set of extremal graphs for $W_{2k+1}$ on $n$ vertices, and $B$ is the property defined on $\mathfrak{U}$ such that if $G\in \mathfrak{U}_i$, then we have $G\in\mathcal{K}_i^k$.

\subsection{Decomposition families of graphs}\label{decomposition family}

Given two graphs $G$ and $H$, denote by $G+ H$ the graph obtained from $G\cup H$, the vertex-disjoint union of graphs $G$ and $H$, by adding edges between each vertex of $G$ and each vertex of $H$. Denote by $K_t$ the complete graph on $t$ vertices and $\overline{K}_t$ the complement of $K_t$.  For every  graph $L$,  Simonovits \cite{Simonovits1974} defined the decomposition family of $L$.

\medskip
\begin{definition}(Simonovits \cite{Simonovits1974}).
\emph{Given a graph $L$, let $\mathcal{F}:=\mathcal{F}(L)$ be the family of minimal graphs $F$ that satisfy the following: there exists a constant $t=t(L)$ depending on $L$ such that $L\subset (F\cup \overline{K}_t)+ T(t,p-1)$. We call $\mathcal{F}$ the {\it decomposition family} of $L$.}
\end{definition}

The decomposition family of a non-bipartite graph often helps us to determine the error term of its Tur\'{a}n number (see Theorem~\ref{p+1 chromatic}). More precisely, the decomposition family of a non-bipartite graph helps us to determine the fine structure of its extremal graphs. A deep theorem of Simonovits \cite{Simonovits1974} shows that if the decomposition family $\mathcal{F}(L)$ of $L$ contains a linear forest, then the extremal graphs for $L$ have very simple and symmetric structure. In our case, the decomposition family of $W_{2k+1}$ is $\{S_{k+1},C_{2k}\}$. Thus we can not characterize the extremal graphs for $W_{2k+1}$  by directly applying Simonovits' theorem.

\subsection{Other lemmas}

For a graph $G$, let $V(G)$ be the vertex set of $G$ and $e(G)$ be the number of edges $G$. Let $G[X]$ be the subgraph of $G$ induced by $X\subseteq V(G)$.    Denote by $k\cdot G$ the vertex-disjoint union of $k$ copies of $G$. Erd\H{o}s and Gallai \cite{erdHos1959maximal} proved the following famous result.

\begin{theorem}(Erd\H{o}s and Gallai \cite{erdHos1959maximal}).\label{Pathk1}
 Let $G$ be a graph on $n$ vertices. If $G$ does not contain a path on $k$ vertices and $n\ge k\ge 2$, then $e(G)\leq (k-2)n/2$.
\end{theorem}

We will need the following lemmas.

\begin{lemma}\label{bound degree path}
Let $k>k^\prime\geq 0$ and $\epsilon>0$. Assume that $G$ is a graph on $n$ vertices with $\Delta(G)\leq k-1$ and $e(G)=(k^\prime n+\epsilon n)/2$. There exists a constant $n_0=n_0(k,k^\prime,s,\epsilon)$  such that if $n\geq n_0$, then $G$ contains $s\cdot P_{k^\prime+2}$ and $s\cdot S_{k^\prime+2}$ as subgraphs.
\end{lemma}
\begin{proof}
We only prove that $G$ contains $s\cdot P_{k^\prime+2}$ as a subgraph. Let $n_0\geq2(k^\prime+2)ks/\epsilon$. By Theorem~\ref{Pathk1}, $G$ contains $P_{k^\prime+2}$ as a subgraph. Let $t=\max\{\ell:\ell\cdot P_{k^\prime+2}\subseteq G\}$. If $t\geq s$, then we are done. Suppose that $t\leq s-1$. Let $G^\prime=G-V(t\cdot P_{k^\prime+2})$. Since $n_0\geq2(k^\prime+2)ks/\epsilon$, we have
 \begin{align*}
e(G^\prime)&\geq \frac{k^\prime n+\epsilon n}{2}-(k-1)(t-1)(k^\prime+2)\\
&> \frac{1}{2} k^\prime(n-(t-1)(k^\prime+2)).
\end{align*}
It follows from Theorem~\ref{Pathk1} that $G^\prime$ contains $P_{k^\prime+2}$ as a subgraph, a contradiction to the maximality of $t$. Thus $G$ contains $s\cdot P_{k^\prime+2}$ as a subgraph. The proof is complete.\end{proof}


\begin{lemma}\label{Main Lemma 1}
Let $G$ be a graph with a vertex partition $V(G)=V_0\cup V_1$ such that $|V_0|=n_0$, $|V_1|=n_1$, and $n_0\geq n_1\geq n_0/2$. Let $G_0=G[V_0]$ with $e(G_0)\leq \lfloor(k-1-k_1) n_0/2 \rfloor+N$ and $G_1=G[V_1]$ with $\Delta(G_1)=k_1\leq k-1$, where $N$ is a constant depending on $k$. Suppose $n_0$ is sufficiently large and
\begin{equation}\label{1}
e(G)\geq n_0 n_1 + \left\lfloor\frac{(k-1)n_0}{2}\right\rfloor +1.
\end{equation}
In each of the following cases:\\
(\romannumeral1) $k_1=k-1$,\\
(\romannumeral2) $k_1=2$ and $k\geq 4,$\\
(\romannumeral3) $k_1=3$ and $k=5$,\\
we have either $G$ contains a copy of $W_{2k+1}$ or $G\in  K_{n_0,n_1}(\mathcal{U}^k_{n_0}(P_{2k-1}),K_2)$.
\end{lemma}
\begin{proof}  Since  $n_0\geq n_1$ and  $e(G_0)\leq \lfloor(k-1-k_1) n_0/2\rfloor +N$, by (\ref{1}) we have $e(G_1)> \lfloor(k_1-1+\epsilon)n_1/2\rfloor$, where $0<\epsilon<1$. As $n_1\geq n_0/2$ is sufficiently large, by Lemma~\ref{bound degree path}, $G_1$ contains $2(N+1)\cdot S_{k_1+1}$ as a subgraph.

\medskip

\noindent($\romannumeral1$) $k_1=k-1$. Then $G_1$ contains $2(N+1)\cdot S_{k}$ as a subgraph. First, we show that there is a non-edge between $G_0$ and each $S_{k}$ in $G_1$. In fact, if there is an $S_{k}$ in $G_1$ such that each vertex of it is adjacent to each vertex of $G_0$, then we have $e(G_0)\leq 1 $. Otherwise, $G$ contains $W_{2k+1}$ as a subgraph and we are done. Hence, if either $n_0=n_1$, or $n_0=n_1+1$ is odd and $k=2$, then, combining $e(G_0)\leq 1$, $\Delta(G_1)\leq k-1$ and (\ref{1}), we have $G\in K_{n_0,n_1}(\mathcal{U}^k_{n_0}(P_{2k-1}),K_2)$ and we are done. For the rest cases, we have $e(G)\leq \lfloor(k-1)n_1/2\rfloor+n_0n_1+1<\lfloor(k-1)n_0/2\rfloor+n_0n_1+1,$ a contradiction. Thus there is a non-edge between $G_0$ and each $S_{k}$ in $G_1$. Hence, we have
\begin{align*}
e(G)&\leq N+n_0n_1+\left\lfloor\frac{(k-1)n_1}{2}\right\rfloor-2(N+1)\\
&<\left\lfloor\frac{(k-1)n_0}{2}\right\rfloor+n_0n_1+1,
\end{align*}
a contradiction to (\ref{1}). We finish the proof of the lemma for $k_1=k-1$.

\medskip

\noindent($\romannumeral2$) $k_1=2$ and $k\geq 4$. Then $G_1$ contains $2(N+1)\cdot S_{3}$ as a subgraph. If each vertex of $G_0$ is adjacent to each vertex of  $2\cdot S_3\cup \overline{K}_{k-4}$ in $G_1$, then we will find a $W_{2k+1}$ as a subgraph in $G$. Since
$$e(G_0)\geq n_0n_1 + \left\lfloor\frac{(k-1)n_0}{2}\right\rfloor+1-  n_0n_1- n_1\geq \left\lfloor\frac{(k-3)n_0}{2}\right\rfloor+1,$$
$G_0$ contains a copy of $S_{k-1}$. Hence the subgraph of $G$ induced by $V(2\cdot S_3 \cup K_{k-4})\cup V(S_{k-1})$ contains a copy of $W_{2k+1}$ and we are done. Thus there is a non-edge between $G_0$ and each  $2\cdot S_3\cup \overline{K}_{k-4}$ in $G_1$. Hence, we have $e(G)\leq n_0n_1+  \lfloor(k-3)n_0/2\rfloor+N + n_1  -(N+1)< \lfloor(k-1)n_0/2\rfloor+n_0n_1+1 $, a contradiction. The proof of this case is complete.

\medskip

\noindent($\romannumeral3$) $k_1=3$ and $k=5$. Then $G_1$ contains $2(N+1)\cdot S_{4}$ as a subgraph. Suppose that each vertex of $G_0$ is adjacent to each vertex of $S_{4}$ in $G_1$. Since
$$e(G_0)\geq n_0n_1 + 2n_0+1-  n_0n_1- \left\lfloor \frac{3n_1}{2} \right\rfloor \geq \left\lceil\frac{n_0}{2}\right\rceil +1$$
and $\Delta(G_0)\leq4$, $G_0$ contains a copy of $S_3\cup K_2\cup K_2$.  Therefore, the induced subgraph of $G$ on vertex set $V(S_4)\cup V(S_3\cup K_2\cup K_2)$ contains $W_{11}$ as a subgraph and we are done. Thus there is a non-edge between $G_0$ and each $S_{4}$ in $G_1$. Now, we have $e(G)\leq n_0n_1+ \lfloor n_0/2 \rfloor+N+ \lfloor 3n_1/2 \rfloor -2(N+1)<n_0 n_1 + 2n_0+1$, a contradiction. The lemma is proved.\end{proof}

The set of neighbors of $v$ is denoted by $N_G(v)$. Let $d_{G}(v)=|N_G(v)|$ be the {\it degree} of $v$. For a subgraph $G^\prime \subseteq G$, let  $N_{G^\prime}(v)=N_G(v)\cap V(G^\prime)$ and  $d_{G^\prime}(v)=|N_{G^\prime}(v)|$.

\begin{lemma}\label{main lemma}
Suppose that $G$ is a $W_{2k+1}$-free graph with a vertex partition $V(G)=V_0\cup V_1$ such that $|V_0|=n_0$, $|V_1|=n_1$ and $n_0\geq n_1\geq k+1$. Assume $G_0=G[V_0]$ with $\Delta(G_0)\leq k-1$ and $G_1=G[V_1]$ with $\Delta(G_1)\leq k-1$. If $k\geq 2$ and $n_0$ is sufficiently large, then
$$
e(G)\leq n_0\cdot n_1 + \left\lfloor\frac{(k-1)n_0}{2}\right\rfloor +1.
$$
Moreover, if the equality holds, then $G\in K_{n_0,n_1}(\mathcal{U}^k_{n_0}(P_{2k-1}),K_2)$.
\end{lemma}
\begin{proof}
We will prove the lemma by induction on $n_1$. Suppose that (\ref{1}) holds and $G$ does not contain $W_{2k+1}$ as a subgraph. It will be shown that $G\in K_{n_0,n_1}(\mathcal{U}^k_{n_0}(P_{2k-1}),K_2)$. We first prove the base case where $n_1=k+1$. Let $V_1=\{x_1,\ldots,x_{k+1}\}$. It is sufficient to prove that $G_1=K_2\cup \overline{K}_{k-1}$. Suppose that either $\Delta(G_1)\geq 2$ or $G_1$ contains a copy of $2\cdot K_2$.
Let $$X=\bigcap_{u\in V_1}N_{G_0}(u).$$
By (\ref{1}) and $\Delta(G_0)\leq k-1$, we have that the number of  non-edges between $V_0$ and $V_1$ is at most ${k+1 \choose 2}$. Thus
$$|X|\geq n_0-{k+1 \choose 2}.$$
Moreover, we have $\Delta(G[X])\leq k-2$. Otherwise, $G$ contains $W_{2k+1}$ as a subgraph, a contradiction. In fact, if there is a vertex $y$ in $X$ with $d_{G[X]}(y)=k-1$, then the induced subgraph of $G$ on vertex set $\{y\}\cup N_{G[X]}(y)\cup V_1$ contains $W_{2k+1}$ as a subgraph (note that $\Delta(G_1)\geq 2$ or $G_1$ contains a copy of $2\cdot K_2$).
Since $n_0$ is sufficiently large, we have
\begin{align*}
e(G)&<(n_0-|X|)(k-1)+\left\lfloor\frac{(k-2)|X|}{2}\right\rfloor+n_0n_1+{k+1 \choose 2}\\
&<n_0n_1+\left\lfloor\frac{(k-1)n_0}{2}\right\rfloor+1,
\end{align*}
a contradiction to (\ref{1}). Now, suppose that the lemma holds for $n_1-1$.

\medskip

\noindent{\bf Claim 1.} For each vertex $x\in V_1$, we have $d_{G}(x)\geq n_0+1$.

\medskip

\begin{proof} If there is a vertex $x\in V_1$ with $d_{G}(x)<n_0+1$, then $e(G-\{x\})\geq n_0(n_1-1)+\lfloor(k-1)n_0/2\rfloor+1$.  Since $G-\{x\}$ does not contain $W_{2k+1}$ as a subgraph, by the induction hypothesis, $G-\{x\}\in K_{n_0,n_1-1}(\mathcal{U}^{\prime}_{n_0,k},K_2)$ and $d_{G}(x)=n_0$. Hence we have $G\in K_{n_0,n_1}(\mathcal{U}^k_{n_0}(P_{2k-1}),K_2)$ by an easy observation.\end{proof}

 Let $\Delta(G_1)=k_1$. We may suppose that $k_1\geq 1$. Otherwise, there is nothing to prove. By Claim 1, for each vertex $x\in V_1$, we have
\begin{equation}\label{eq 2}
d_{G_0}(x)\geq n_0+1-k_1.
\end{equation}
Assume $x\in V_1$ with $d_{G_1}(x)=k_1$. Let $N_{G_1}(x)=\{x_1,\ldots,x_{k_1}\}$ and
 $$X=\bigcap_{u\in\{x,x_1,\ldots,x_{k_1}\}}N_{G_0}(u).$$
 It follows from (\ref{eq 2}) that $$|X|\geq n_0-(k_1+1)(k_1-1).$$
Let $\epsilon<1$. Then we have
$$e(G[X])>\left\lfloor\frac{(k-k_1-2+\epsilon)|X|}{2}\right\rfloor.$$
Otherwise, since $n_0$ is sufficiently large, we have
\begin{align*}
e(G)&\leq\frac{(k-k_1-2+\epsilon)|X|}{2}+(n_0-|X|)(k-1)+\left\lfloor\frac{k_1n_1}{2}\right\rfloor+n_0n_1\\
&<\frac{(k-1)n_0}{2}+n_0n_1+1,
\end{align*}
a contradiction.
Since $n_0$ is sufficiently large, by Lemma~\ref{bound degree path}, $G[X]$ contains $k_1\cdot P_{k-k_1}$ as a subgraph. Let $Y=k_1\cdot P_{k-k_1}$. If either $3\leq k_1\leq k-3$ or $4\leq k_1=k-2$,  then $k_1(k-k_1)+k_1\geq 2k$. Hence,  the subgraph of $G$ induced by  $\{x\}\cup N_{G_1}(x)\cup V(Y)$ contains $W_{2k+1}$ as a subgraph (map the center of the wheel to $x$). We next consider the leftover cases:

\medskip

\noindent{\bf Case 1.} $k_1=k-1$, i.e., $\Delta(G_1)=k-1$. Then there is a vertex $x\in V_1$ with $d_{G_1}(x)=k-1$. Let $N_{G_1}(x)=\{x_1,\ldots,x_{k-1}\}$ and
 $$X=\bigcap_{u\in\{x,x_1,\ldots,x_{k-1}\}}N_{G_0}(u).$$
By (\ref{eq 2}), we have
$$|X|\geq n_0-k(k-2).$$
Moreover, we have $e(G[X])\leq1$. Otherwise $G$ contains $W_{2k+1}$ as a subgraph. Hence we have $e(G_0)\leq k(k-1)(k-2)+1.$ If $n_1< n_0/2$, since $n_0$ is sufficiently large, then
\begin{align*}
e(G)&\leq k(k-1)(k-2)+1+\left\lfloor\frac{(k-1)n_1}{2}\right\rfloor+n_0n_1\\
&\leq k(k-1)(k-2)+1+\left\lfloor\frac{(k-1)n_0}{4}\right\rfloor+n_0n_1\\
&<\frac{(k-1)n_0}{2}+n_0n_1+1,
\end{align*}
a contradiction to (\ref{1}). If $n_1\geq n_0/2$, then the result follows from Lemma~\ref{Main Lemma 1}.

\medskip

\noindent{\bf Case 2.} $k_1=1$, i.e., $\Delta(G_1)=1$. Let $x$ be a vertex in $G_1$. By Claim 1, we have $d_{G_1}(x)=1$  and $d_{G_0}(x)=n_0$. Hence we have $\Delta(G_0)\leq k-2$. Otherwise, there is a vertex $y\in G_0$ with $|N_{G_0}(y)|=k-1$, and the subgraph of $G$ induced by $\{y\}\cup N_{G_0}(y)\cup V_1$ contains $W_{2k+1}$ as a subgraph (note that $G_1$ contains $2\cdot K_2$ as a subgraph). Thus
\begin{align*}
e(G)&\leq n_0n_1+\left\lfloor\frac{(k-2)n_0}{2}\right\rfloor+\left\lfloor\frac{n_1}{2}\right\rfloor\\
&<\frac{(k-1)n_0}{2}+n_0n_1+1,
\end{align*}
a contradiction.

\medskip

\noindent{\bf Case 3.} $k_1=2$ and $k\geq4$, i.e., $\Delta(G_1)=2$. As $d_{G}(y)\geq n_0+1$ for any $y\in V_1$, there are no isolated vertices in $G_1$. We consider the following two subcases:

\medskip

{\bf Subcase 3.1.} $G_1$ contains either $2\cdot P_3$ or $P_4\cup K_2$ as a subgraph. Let $x_1,\ldots,x_6$ be the vertices of $2\cdot P_3$ or $P_4\cup K_2$. We choose arbitrary $k-4$ vertices, say $x_7,\ldots,x_{k+2}$, from $V(G_1)\setminus \{x_1,\ldots,x_6\}$. Let
 $$X=\bigcap_{u\in\{x_1,\ldots,x_{k+2}\}}N_{G_0}(u).$$
Since $d_{G}(y)\geq n_0+1$ and $\Delta(G_1)=2$, we have
$$|X|\geq n_0-(k+2).$$
Moreover, we have $\Delta(G[X])\leq k-3$. Otherwise, if there is a vertex $x\in X$ with $d_{G[X]}(x)\geq k-2$, then the  subgraph of $G$ induced by $\{x\}\cup N_{G[X]}(x)\cup \{x_1,\ldots,x_{k+2}\}$ contains $W_{2k+1}$ as a subgraph, a contradiction. If $n_1<n_0/2$, then
$$e(G)\leq \lfloor(k-3)|X|/2\rfloor+(k+2)(k-1)+n_1+n_0n_1<\lfloor(k-1)n_0/2\rfloor+n_0n_1+1,$$
a contradiction. If $n_1\geq n_0/2$, by Lemma~\ref{Main Lemma 1},  then we have $G\in K_{n_0,n_1}(\mathcal{U}^k_{n_0}(P_{2k-1}),K_2)$.

\medskip

{\bf Subcase 3.2.} $G_1$ does not contain $2\cdot P_3$ or $P_4\cup K_2$ as a subgraph. It is easy to see that $P_3\cup (n_1-3)/2\cdot K_2   \subseteq G_1\subseteq K_3\cup (n_1-3)/2 \cdot K_2$ and $n_1$ is odd (recall that there are no isolated vertices in $G_1$). Let $P_3$ be a subgraph of $G_1$ with $V(P_3)=\{x_1,x_2,x_3\}$ and
 $$X=\bigcap_{u\in\{x_1,x_2,x_3\}}N_{G_0}(u).$$
Since $d_{G_0}(y)\geq n_0+1$, we have
$$|X|\geq n_0-3.$$
Moreover, we have $e(G[X])\leq \lfloor(k-3+\epsilon)|X|/2\rfloor$, where $0<\epsilon<1/2$. Otherwise, by Lemma~\ref{bound degree path}, $G[X]$ contains $2\cdot P_{k-1}$ as a subgraph. Hence $G$ contains $ W_{2k+1}$ as a subgraph, a contradiction. Therefore, by $|X|\geq n_0-3$, we have
\begin{align*}
e(G)&\leq \left(3(k-1)+\left\lfloor\frac{(k-3+\epsilon)|X|}{2}\right\rfloor\right)+\left(\left\lceil\frac{n_1}{2}\right\rceil+1\right)+n_0n_1\\
&<\left\lfloor\frac{(k-1)n_0}{2}\right\rfloor+n_0n_1+1,
\end{align*}
a contradiction.

\medskip

\noindent{\bf Case 4.} $k_1=k-2=3$, i.e, $\Delta(G_1)=3$ and $k=5$. Let $x\in V_1$ with $d_{G_1}(x)=3$. We assume $N_{G_1}(x)=\{x_1,x_2,x_3\}$ and define
 $$X=\bigcap_{u\in\{x,x_1,x_2,x_3\}}N_{G_0}(u).$$
Since $d_{G_0}(y)\geq n_0+1$, we have
$$|X|\geq n_0-8.$$
We claim that $G[X]$ does not contain $2\cdot P_3$ as a subgraph. Otherwise, the subgraph of $G$  induced by $V(2\cdot P_3)\cup \{x,y,x_1,x_2,x_3\}$ contains a copy of $W_{11}$ in $G$, where $y\in X$, a contradiction. Hence, by $\Delta(G_0)\leq 4$, we have $e(G[X])\leq 4\times 3 + \lfloor(|X|-3)/2\rfloor\leq \lfloor|X|/2\rfloor+6$. If $n_1< n_0/2$, then
\begin{align*}
e(G)&\leq \left\lfloor\frac{|X|}{2}\right\rfloor+6+8\times 4 +\left\lfloor\frac{3n_1}{2}\right\rfloor+n_0n_1\\
&<2n_0+n_0n_1+1,
\end{align*}
 a contradiction. If $n_1\geq n_0/2$, then we have $G\in  K_{n_0,n_1}(\mathcal{U}^k_{n_0}(P_{2k-1}),K_2)$  by Lemma~\ref{Main Lemma 1}. The proof is complete.\end{proof}

\section{Proof of Theorem~\ref{main}}

Recall that $f(n,k)=\max\{n_0(n-n_0)+\frac{(k-1)n_0}{2}+1:n_0+n_1=n\}$. We begin with a proposition about $f(n,k)$.

\begin{proposition}\label{pro f(n,k)}
Let $N$ be an even integer. Then
$$f(n,k)=f(n-2N,k)+t(2N,2)+(n-2N)N+\frac{(k-1)N}{2}$$
and
$$f(n,k)> f(n-1,k)+\frac{n}{2}+\frac{k}{4}-1.$$
\end{proposition}

\begin{proof} A basic calculation shows that $f(n,k)$ attains its maximum when

$$
n_0=\left\{
  \begin{array}{ll}
    \frac{2n+k-1}{4}, & \hbox{for $2n+k-1 \equiv0$ mod $4$;} \\
   \frac{2n+k-2}{4}, & \hbox{for $2n+k-1 \equiv1$ mod $4$;} \\
    \frac{2n+k-3}{4},\frac{2n+k+1}{4} & \hbox{for $2n+k-1 \equiv2$ mod $4$;} \\
    \frac{2n+k}{4}, & \hbox{for $2n+k-1 \equiv3$ mod $4$.}
  \end{array}
\right.
$$
\noindent Since $n-2N \equiv n$ mod 4, we have \begin{eqnarray*}
f(n,k)&=&n_0(n-n_0)+\frac{(k-1)n_0}{2}+1\\
&=&(n_0-N)(n-n_0-N)+(n-N)N+\frac{(k-1)(n_0-N)}{2} +\frac{(k-1)N}{2}+1\\
&=&f(n-2N)+t(2N,2)+(n-2N)N+\frac{(k-1)N}{2}.
\end{eqnarray*}
If $2n+k-1 \equiv 0$ mod 4, then

\begin{eqnarray*}
f(n,k)-f(n-1,k)&=& \frac{2n+k-1}{4}\left(n-\frac{2n+k-1}{4}\right)+\frac{(k-1)}{2}\frac{2n+k-1}{4}+1\\
&&-\frac{2(n-1)+k+1}{4}\left(n-1-\frac{2(n-1)+k+1}{4}\right)\\
&&-\frac{(k-1)}{2}\frac{2(n-1)+k+1}{4}-1\\
&=&\frac{n}{2}+\frac{k}{4}-\frac{1}{4}\\
&>&\frac{n}{2}+\frac{k}{4}-1.
\end{eqnarray*}
Similarly, we can prove the proposition when $2n+k-1 \equiv i$ mod 4 for $i\neq 0$.\end{proof}

Let $k\geq 2$. Let $\mathfrak{U}_n$ be the set of $n$-vertex extremal graphs for $W_{2k+1}$. We choose an arbitrary graph $L_n\in \mathfrak{U}_n$. Since the graphs in $\mathcal{K}_n^k$ are $W_{2k+1}$-free, we have
\begin{equation}\label{5}e(L_n)\geq f(n,k).\end{equation}
We define
\begin{equation}\label{6}
\phi(L_n)=e(L_n)-f(n,k).
\end{equation}
Then, $\phi(L_n)$ is a non-negative integer. We will  prove the theorem  by the progressive induction, where $B$ is the property defined on $\mathfrak{U}=\cup_{i=1}^\infty\mathfrak{U}_i$ such that if $G\in \mathfrak{U}_i$, then we have $G\in\mathcal{K}_i^k$. According to the lemma of progressive induction, it is enough to show that if $L_n\notin \mathcal{K}_n^k$, then there exists an $n^\prime$ such that $n/2<n^{\prime}<n$ and $\phi(n^{\prime})>\phi(n)$ provided $n$ is sufficiently large.

By Theorem~\ref{erdos-stone} and (\ref{5}), there is an $n_1$ such that if $n>n_1$, then  $L_n$ contains $T(2N,2)$ as a subgraph, where $N$ is an even large constant depending on Lemma~\ref{main lemma}. We define $\epsilon=\frac{2k}{N}$. Let $B_1$ and $B_2 $ be two partite sets of $T(2N,2)$.  Since $L_n$ does not contain $W_{2k+1}$ as a subgraph, we have $\Delta(L_n[B_i])\leq k-1$ for $i=1,2$. Let $A=V(L_n)-V(T(2N,2)$. If $x\in A$, then there exists an $i(x)\in\{1,2\}$ such that $x$ is adjacent to less than $k$ vertices of $B_{i(x)}$. Indeed, if $x$ is adjacent to at least $k$ vertices of each $B_{i}$, then the subgraph of $L_n$ induced by $x\cup N_{L_n}(x)$ contains $W_{2k+1}$ as a subgraph, a contradiction. We next partition $A$ into three subsets $C_1$, $C_2$ and $D$. If $x$ is adjacent to less than $k$ vertices of $B_i$ and more than $(1-\epsilon)N$ vertices of $B_{3-i}$, then we put $x\in C_i$ for $i=1,2$. If $x\in D$, then $x$ is adjacent to at most $(1-\epsilon)N$ vertices of $B_i$ for $i=1,2$.

Moreover, we claim that $\Delta(L_n[B_i\cup C_i])\leq k-1$. Suppose that there is a vertex $x\in B_i\cup C_i$ with $d_{L_n[B_i\cup C_i]}(x)\geq k$. Choose $X=\{x_1,x_2,\ldots,x_k\}\subseteq d_{L_n[B_i\cup C_i]}(x)$. Then the common neighbours of $\{x\}\cup X$ in $B_{3-i}$ is at least $N-(k+1) \epsilon N= N- 2k^2-2k>k$ (recall that $N$ is a large constant). Thus  $L_n$ contains $W_{2k+1}$ as a subgraph, a contradiction.

Denote by $e_L$ the number of edges joining $A$ and $B_1\cup B_2$. We have
\begin{equation}\label{8}
e(L_n)=e(L_n[A])+e_L+e(L_n[B_1\cup B_2]).
\end{equation}

We choose an $L^\prime_n\in \mathcal{K}_n^k$ such that there is a component on $N$ vertices in the larger partite set of $L_n^\prime$ (since $N\geq 2k$ and $n$ is sufficiently large, there exists such a graph in $\mathcal{K}_n^k$ by Proposition~\ref{k-1 regular no C2k}). Hence we can choose two vertex sets $B^\prime_1$ and $B^\prime_2$ of different classes of $L^\prime_n$ such that $L_n^\prime[B^\prime_1]\in \mathcal{U}^k_{N}(P_{2k-1})$ and $L_n^\prime[B^\prime_2]=\overline{K}_N$. Let $A^\prime=V(L^\prime_n)\setminus (B^\prime_1\cup B^\prime_2).$ Denote by $e_{L^\prime}$ the number of edges joining $A^\prime$ and $B^\prime_1\cup B^\prime_2$. By Proposition~\ref{pro f(n,k)}, we have
\begin{equation}\label{9}
e(L^\prime_n)=e(L^\prime_{n}[A^\prime])+e_{L^\prime}+e(L_N^\prime[B^\prime_1\cup B^\prime_2])=f(n-2N,k)+e_{L^\prime}+e(L_N^\prime[B^\prime_1\cup B^\prime_2]).
\end{equation}
Obviously, we have $e_{L^\prime}=(n-2N)N$.

Since $\Delta(L_n[B_i])\leq k-1$ for $i=1,2$ and $L_n[B_1\cup B_2]$ does not contain $W_{2k+1}$ as a subgraph, by Lemma~\ref{main lemma}, we have
\begin{equation}\label{8.1}
e(L_n[B_1\cup B_2])\leq e(L^\prime_n[B^\prime_1\cup B^\prime_2])+1.
\end{equation}
Let $L_{n-2N}$ be an extremal graph for $W_{2k+1}$ on $n-2N$ vertices. Since $L_n[A]$ is $W_{2k+1}$-free on $n-2N$ vertices, we have
\begin{equation}\label{8.2}
e(L_n[A])\leq e(L_{n-2N}).
\end{equation}
Combining (\ref{8}), (\ref{9}), (\ref{8.1}), and (\ref{8.2}), we  have
\begin{align*}
\phi(n)&=e(L_n)-e(L^\prime_n)
\\&=e(L_n[A])-f(n-2N,k)+(e_{L}-e_{L^\prime})+e(L_n[B_1\cup B_2])-e(L_N^\prime[B^\prime_1\cup B^\prime_2])
\\&\leq e(L_{n-2N})-f(n-2N,k)+(e_{L}-e_{L^\prime})+1
\\&=(e_{L}-e_{L^\prime})+\phi(n-2N)+1.
\end{align*}
Thus
\begin{equation}\label{10}
\phi(n)\leq(e_{L}-e_{L^\prime})+\phi(n-2N)+1.
\end{equation}
It will be proved that if $n$ is large enough, then\\
(a) either $\phi(n)<\phi(n-2N)$;\\
(b) or $\phi(n)<\phi(n-1)$;\\
(c) or $L_n\in \mathcal{K}_n^k$.\\
This will complete our proof.

If there is a vertex $x\in L_n$ with $d_{L_n}(x)<n/2+k/4-1$, then we claim $\phi(n)<\phi(n-1)$. In fact, the subgraph $L_n^\ast=L_n-\{x\}$ is $W_{2k+1}$-free. Thus $e(L_n)-d_{L_n}(x)=e(L_n^\ast)\leq e(L_{n-1})$ and from this
$e(L_n)-e(L_{n-1})\leq d_{L_n}(x)<n/2+k/4-1$. By Proposition~\ref{pro f(n,k)}, $e(L^\prime_n)-e(L^\prime_{n-1})=f(n,k)-f(n-1,k)\geq n/2+k/4-1$, we have $\phi(n)=e(L_n)-e(L^\prime_n)< e(L_{n-1})-e(L^\prime_{n-1})=\phi(n-1)$.

Suppose now that neither (a) nor (b) holds. Then  each $x\in L_n$ has degree at least $n/2+k/4-1$  and $\phi(n)\geq \phi(n-2N)$.
From (\ref{10}), we have $0\leq\phi(n)-\phi(n-2N)\leq e_L-e_{L^\prime}+1$.

We shall prove $L_n\in \mathcal{K}_n^k$.

\medskip

\noindent{\bf Claim 1:} There exists an $N_1$ such that $|D|\leq N_1$.

\medskip

\begin{proof} First recall that $\Delta(G[B_i\cup C_i])\leq k-1$ and  $d_{G[B_1\cup B_2]}(x)\leq k+(1-\epsilon)N\leq N-k$ for each $x\in D$. Therefore, the number of edges joining $B_i$ and $C_i$ is at most $N(k-1)$ and then
\begin{equation*}\label{11}
e_{L}\leq (n-2N-|D|)N+2N(k-1)+|D|(N-k)=e_{L^\prime}+2N(k-1)-|D|k.
\end{equation*}
We have
$$|D|\leq \frac{1}{k}(e_{L^\prime}-e_L+2N(k-1))\leq\frac{1}{k}(2N(k-1)+1)= N_1,$$ and hence the proof is complete.\end{proof}

\noindent{\bf Claim 2:} Each vertex from $B_i\cup C_i$ is adjacent to at most  $k-1$ other vertices of $B_i \cup C_i$ for $i=1,2$.

\medskip

\begin{proof} This claim was already proved.\end{proof}

\noindent{\bf Claim 3:} For $i=1,2$, $|B_i\cup C_i|=n/2+O(\sqrt{n}).$

\medskip

\begin{proof} In order to show this, we omit the edges inside $B_i\cup C_i$ ($i=1,2$) and the edges incident to $D$. Then the resulting graph
is $L_{n-|D|}$ which is $2$-chromatic and has $t(n,2)-O(n)$ edges. Thus there is a constant $N_2$ such that
$$\left||B_i\cup C_i|-\frac{n}{2}\right|\leq N_2 \sqrt{n},$$ and the proof is complete.\end{proof}

\noindent{\bf Claim 4:} For $i=1,2$, there is a constant $N_3$ such that every $x\in B_i\cup C_i$ is adjacent to all but at most $N_3 \sqrt{n}$ vertices of $L_n -(B_i\cup C_i)$.

\medskip

\begin{proof} This follows immediately from the fact that $d_{G[B_i\cup C_i]}(x)\leq k-1$, $ d_{L_n}(x)\geq n/2+k/4-1 $ and Claim 3.\end{proof}

For $i=1,2$, let $D_i$ be the set of  vertices which have at most $k-1$ neighbours in $B_i\cup C_i$.

\medskip

\noindent{\bf Claim 5:} $D$ is the disjoint union of $D_1$ and $D_2$.

\medskip

\begin{proof} In fact, if $x\in D$, then there is an $i(x)$ such that $x$ is adjacent to at least $(1/3)(n/2)$ vertices of $B_{i(x)}\cup C_{i(x)}$. Otherwise $d_{L_n}(x)<O(\sqrt{n})+(2/3)(n/2)<n/2 $, a contradiction. Furthermore, $x$ is adjacent to less than $k$ vertices of $B_{3-i(x)}\cup C_{3-i(x)}$. Otherwise $L_n$ contains $W_{2k+1}$ as a subgraph. In fact, without loss of generality, let $i(x)=2$, we select $k$ vertices of $B_1\cup C_1$ joining to $x$. Then select $k$ vertices in $B_2\cup C_2$ joining to $x$ and to the $k$ vertices already chosen in $B_1\cup C_1$. This is possible since each vertex selected from $B_1\cup C_1$ is adjacent to at least $(n/2)-O(\sqrt{n})$ of $B_{2}\cup C_{2}$ and $x$ is adjacent to at least $(1/3)(n/2)$ vertices of $B_{2}\cup C_{2}$. Thus
$G$ contains a copy of $W_{2k+1}$, a contradiction. We finish the proof of this claim.\end{proof}

Let $V_1=B_1\cup C_1 \cup D_1$ and $V_2=B_2\cup C_2 \cup D_2$. By Claim 5, $V_1\cup V_2$ is a vertex partition of $L_n$. Note that each $x\in L_n$ has degree at least $n/2+k/4-1$ and each vertex in $D_i$ is adjacent to $n/2+O(\sqrt{n})$ vertices of $V_{3-i}$. It is easy to see that $\Delta(L_n[V_i])\leq k-1$ for $i=1,2$. Hence by Lemma~\ref{main lemma} and (\ref{5}), we have $L_n \in \mathcal{K}_n^k$. We finish the proof of the theorem. \hfill $\square$\par\medskip

\remark We can also determine the Tur\'{a}n numbers of $K_1+P_{2k}$ and $K_1+P_{2k+1}$ by similar arguments of the proof of Theorem~\ref{main}.

\section{Proof of Theorem~\ref{MAIN}}\label{proof of NIM}

\subsection{Lemma for Theorem~\ref{MAIN}}

In this section, we will consider a complete graph $K_n$ with a red/blue-edge-coloring. For any two vertices $u$ and $v$, call $u$ a red (blue) neighbor of $v$ if the edge $uv$ is colored red (blue). If an edge $e$ is not contained in any monochromatic copy of a given graph $H$, then we call $e$ being \textit{NIM-$H$}. If a subgraph $G$ of $K_n$ consists of NIM-$H$ edges, then we call $G$ being NIM-$H$.

We need the following well-known theorems. The first one due to Ramsey is one of the most important results in combinatorics.

\begin{theorem}(Ramsey).\label{Ramsey} For every $t$ there exists $N = R(t)$ such that every $2$-coloring of the edges of
$K_N$ has a monochromatic $K_t$ subgraph.
\end{theorem}

In 1954, K\"{o}v\'{a}ri, S\'{o}s, and Tur\'{a}n proved the following theorem.

\begin{theorem}(K\"{o}v\'{a}ri, S\'{o}s, and Tur\'{a}n \cite{Kovari1954}).\label{KST}
$$\emph{ex}(n,K_{t,t})=O(n^{2-\frac{1}{t}}). $$
\end{theorem}

Now we are ready to prove the following lemma. Although this lemma is enough for the proof of our theorem with $p=2$, we prove it with $p\geq2$ for further research.

\begin{lemma}\label{Lemma-for-monochromatic-T(n,p)}
Let $p\geq 2$, $t\geq 1$ be given integers and $H$ be a given graph. Then there exists an integer $n_0=n_0(p,t,H)$ such that if $G$ is an NIM-$H$ graph on $n\geq n_0$ vertices containing at least $t(n,p)$ edges, then $G$ contains a blue (or red) copy of $T(tp,p)$ such that the edges inside each class are red (or bule), where $t\geq |V(H)|$.
\end{lemma}
\begin{proof} Since $G$ is an NIM-$H$ graph on $n$ vertices containing at least $t(n,p)$ edges, by  Theorem~\ref{erdos-stone}, $G$ contains $T(Np,p)$ as a subgraph with a vertex partition $V_1\cup \ldots \cup V_p$, where $N$ is a large constant depending on Theorems~\ref{Ramsey},~\ref{KST}, and $p$.
\medskip

\noindent{\bf Claim 1.} There exists a constant $m(t)$ depending on $t$ such that for any two disjoint vertex sets $U,V$ of $K_n$ with $|V|=|U|=m(t)$, there is a monochromatic copy of $K_{t,t}$ between $U$ and $V$.

\medskip

\begin{proof} Without lose of generality, suppose that there are at least $\frac{1}{2}m(t)^2$ red edges between $U$ and $V$. Since $\frac{1}{2}m(t)^2\geq O(2m(t))^{2-\frac{1}{t}}$ when $m(t)$ is large, the result follows from Theorem~\ref{KST}.\end{proof}

\noindent{\bf Claim 2.}  Any two monochromatic copies of $K_\ell$ with $\ell\geq m(t)$ in different classes of $T(Np,p)$ have the same color.

\medskip

\begin{proof} Let $\ell\geq m(t)$. Suppose that there are a red copy of $K_\ell$ in $V_1$ and a blue copy of $K_\ell$ in $V_2$. Then it follows from Claim 1 that there is a monochromatic copy of $K_{t,t}$ between the red $K_\ell$ and the blue $K_\ell$. Since $t\geq |V(H)|$, the edges of $K_{t,t}$ are contained in a monochromatic copy of $H$, contradicting that $G$ is an NIM-$H$ graph. The proof is complete.\end{proof}

Applying Theorem~\ref{Ramsey}, there is monochromatic copy of $K_{m_0(t)}$ in each class of $T(Np,p).$ By Claim 2, those $p$ monochromatic copies of $K_{m_0(t)}$ have the same color, say red. Let $G^\prime$ be the subgraph of $G$ induced by those $p$ copies of $K_t$. Then $G^\prime$ has a vertex partition $V(G^\prime)=V^\prime_1\cup \ldots \cup V^\prime_p$ such that $V_i^\prime\subset V_i$ for each $i\in[p]$. We say a pair of disjoint vertex sets is monochromatic (red/blue) if all the edges between them have the same color (red/blue).

We will find a copy of $T(tp,p)$ with the property what we need by defining a sequence of graphs. Let $G^0=G^\prime$. By Claim 1, we can define $G^{i+1}$ from $G^i$ by taking $m_{i+1}(t)$ vertices of each class of $G^i$ such that the number of  monochromatic  pairs in the classes of $G^{i+1}$ is least $i+1$, where $m_{i+1}(t)<m_{i}(t)$ is a sufficiently large constant depending Theorem~\ref{KST}. Note that there are ${p \choose 2}$ pairs of vertex sets between $V_1,\ldots,V_p$. Since the edges between different classes of $G^{p \choose 2}$ are NIM-edges, $G^{p \choose 2}$ has the property needed in the lemma with $t=m_{p \choose 2}(t)$. The proof is complete.\end{proof}

\subsection{Proof of Theorem~\ref{MAIN}}
\begin{proof} Let $\mathfrak{U}_i$ be the set of NIM-$W_{2k+1}$ graphs of $K_i$ with maximum number of NIM-$W_{2k+1}$ edges for $i\in[n]$. We will prove the theorem by progressive induction, where $B$ is the property defined on $\mathfrak{U}=\cup_{i=1}^\infty\mathfrak{U}_i$ such that if $G\in \mathfrak{U}_i$, then we have $e(G)=f(n,k)$. This would imply that all the extremal colorings for Theorem~\ref{MAIN} are the colorings defined at the end of Section~\ref{extrmel-graphs-for-wheel}.

Let $G_n\in \mathfrak{U}_n$. Note that if we color the edges of an extremal $n$-vertex graph for $H$ red and color the complement of it blue, then we can see that $g(n,H)\geq\mbox{ex}(n,H)$ for any $H$ and $n$. Hence, we have $e(G_n)\geq\mbox{ex}(n,W_{2k+1})\geq t(n,2)$. Applying Lemma~\ref{Lemma-for-monochromatic-T(n,p)} with $p=2$, there is a blue copy of $T(2N,2)$ in $G_n$ with a vertex partition $B_1\cup B_2$ such that the edges of $K_n[B_1]$ and $K_n[B_2]$ are red, where $N$ is a sufficiently large constant depending on Lemma~\ref{main lemma}. Let $\epsilon:=\epsilon(N,k)$ be a small constant. The proof is essentially the same as the proof of Theorem~\ref{main}. In the following, we briefly sketch the proof of Theorem~\ref{MAIN}.

Let $A=V(G_n)-V(T(2N,2))$ and $\epsilon$ be a small real number. We shall show that if $x\in A$ and $x$ is adjacent to more than $(1-\epsilon)N$ vertices of $B_i$ in $G_n$, then $x$ is adjacent to at most $k-1$ vertices of $B_{3-i}$ in $K_n$. Let $x\in A$ be adjacent to $B_i$ by more than $(1-\epsilon)N$ edges in $G_n$. Then at least $(1-\epsilon)N-2k$ of them are blue. Otherwise, there is a red copy of $W_{2k+1}$ in $G_n$, a contradiction to that $G_n$ is an NIM-$W_{2k+1}$ graph. Hence, there are at most $k-1$ blue edges between $x$ and $B_{3-i}$ in $K_n$. Otherwise, there is a blue copy of $W_{2k+1}$ in $K_n$ using edges of $G_n$, a contradiction. Thus all other edges between $x$ and $B_{3-i}$  are red and not NIM-edges (note that each edge in $K_n[B_i]$ is red).

We can partition the vertices of $A$ into the following classes:
$C_1,C_2,D$ such that if $x\in C_i$ then $x$ is adjacent to less than $k$ vertices of $B_i$ and more than $(1-\epsilon)N$ vertices of $B_{3-i}$ for $i=1,2$, if $x\in D$ then $x$ is adjacent to at most $(1-\epsilon)N$ vertices of each of two of $B_1,B_2$.

Moreover, similarly to the proof of Theorem~\ref{main}, for each $x\in B_i\cup C_i$, there are at most $k-1$ blue edges between $x$ and $B_i\cup C_i$ for $i=1,2$. Thus all other edges between $x$ and $B_i\cup C_i$ are red and not NIM-edges (note that each red edge in $B_i\cup C_i$ has at least $N-2k$ common neighbours in $B_i$ and each edge in $B_i$ is red).

Applying the progressive induction,  we have the following claims.

\medskip

\noindent{\bf Claim 1:} There exists an $N_1$ such that $|D|\leq N_1$.

\medskip
\noindent{\bf Claim 2:} A vertex belonging to $B_i\cup C_i$ is adjacent at most to $k-1$ other vertices of $B_i \cup C_i$ for $i=1,2$.

\medskip
\noindent{\bf Claim 3:} $|B_i\cup C_i|=n/2+O(\sqrt{n}),$ for $i=1,2$.

\medskip
\noindent{\bf Claim 4:} There is a constant $N_3$ such that every $x\in B_i\cup C_i$ is adjacent to all the vertices of $V(L_n)\setminus(B_i\cup C_i)$ but at most than $N_3 \sqrt{n}$ vertices, $i=1,2$.

\medskip

\noindent{\bf Claim 5:} Let $D_i$ be the class of those vertices which are adjacent to $B_i\cup C_i$ by less than $k$ edges for $i=1,2$. Then $D$ is the disjoint union of $D_1, D_2$.

\medskip

Let $V_i=B_i\cup C_i \cup D_i$. Then, repeating the proof of Theorem~\ref{main}, we have $\Delta(G_n[V_i])\leq k-1$. Moreover, edges in $G_n[V_i]$ are blue. The result follows from Lemma~\ref{Main Lemma 1}, $e(G_n)\geq\mbox{ex}(n,W_{2k+1})$ and progressive induction. Moreover, the extremal colorings are determined.\end{proof}

\end{document}